\documentclass{amsart}
\usepackage{amsmath,amssymb,amsthm,amscd}

\newtheorem{theorem}{Theorem}[section]
\newtheorem{theo}[theorem]{Theorem}

\theoremstyle{defi}

\newtheorem{pro}[theorem]{Problem}

\theoremstyle{rem}

\numberwithin{equation}{section}

\def\p{\partial}
\def\R{\mathbb{R}}
\def\C{\mathbb{C}}

\def\i{\sqrt{-1}}

\def\t{\triangle}

\begin{document}
\title[Donaldson's equation and the complex Monge-Amp\`ere equation]
{Entire solutions of Donaldson's equation}
\author{Weiyong He}
\address{Department of Mathematics, University of Oregon, Eugene, Oregon, 97403}
\email{whe@uoregon.edu}

\thanks{The author is supported in-part by a start-up grant of University of Oregon.}


\date{}


\begin{abstract}We shall construct infinitely many special entire solutions to  Donaldson's equation. We shall also prove a Liouville type theorem for entire solutions of Donaldson's equation. We believe that  all entire solutions of Donaldson's equation have the form of the examples constructed in the paper. 
\end{abstract}

\maketitle

\section{Introduction}
Donaldson \cite{Donaldson} introduced an interesting differential operator when he set up a geometric structure for the space of volume forms on  compact Riemannian manifolds. The Dirichlet problems for Donaldson's operator are considered in \cite{Chen-He, He}.  In this note we shall consider this operator on Euclidean spaces.

For $(t, x)\in\Omega\subset \mathbb{R}\times \R^n $ ($n\geq 1$), let  $u(t, x)$ be a smooth function  such that $\t u>0, u_{tt}>0$. We use $\nabla u, \t u$ to denote derivatives with respect to $x$ and $u_t=\p_t u, u_{tt}=\p^2_tu$  to denote derivatives with respect to $t$. Define
a differential operator $Q$ by
\[
Q(D^2u)=u_{tt}\t u-|\nabla u_t|^2.
\]
This operator is strictly elliptic when $u_{tt}>0, \t u>0 $ and $Q(D^2 u)>0$.  When $n=1$, then \[Q(D^2u)=u_{tt}u_{xx} -u_{xt}^2\]
is a real Monge-Amp\`ere operator. When $n=2$, $Q$ can be viewed as a special case of the complex Monge-Amp\`ere operator. In $x$ direction, we identify $\R^2=\C$ with a coordinate $w$.  In $t$ direction, we take a product by $\R$ with a coordinate $s$ and let $z=t+\i s$. We extend $u$ on $\R\times \R^2$ to $\R^4=\C^2$ by
$u(z, w)=u(t, x)$. Then
\[
Q(D^2u)=4(u_{z\bar z}u_{w\bar w}-u_{z\bar w}u_{w\bar z})
\]
is a complex Monge-Amp\`ere operator.

In this paper we shall consider  entire solutions $u: \R\times \R^n\rightarrow \R$ of 
\begin{equation}\label{E-1}
Q(D^2u)=1.
\end{equation}
One celebrated result, proved by J\"orgens \cite{Jo} in dimension 2 and by Calabi  \cite{Calabi} and Pogorelov \cite{Po} in higher dimensions, is that the only convex  solutions of the real Monge-Amp\`ere equation
\begin{equation}\label{E-2}
\det(f_{ij})=1
\end{equation} on the whole of $\R^n$ are the obvious ones: quadratic functions.
 \begin{theo}[ J\"orgens; Calabi; Pogorelov] \label{T-1}Let $f$ be a global convex viscosity solution of \eqref{E-2} in whole $\R^n$. Then $f$ has to be a quadratic function.
 \end{theo}
One can also ask similar questions for the complex Monge-Amp\`ere equations for plurisubharmonic functions. Let $v: \C^n\rightarrow \R$ be a strictly plurisubharmonic function such that $(v_{i\bar j})>0$, which satisfies
\begin{equation}\label{E-3}
\det(v_{i\bar j})=1.
\end{equation}

The similar results as Theorem \ref{T-1} for  the complex Monge-Amp\`ere equation \eqref{E-3} or  Donaldson's equation \eqref{E-1} ($n>1$) are not known. 
For the complex Monge-Amp\`ere equation, LeBrun \cite{LeBrun} investigated the Euclidean Taub-NUT metric constructed by 
 Hawking \cite{Hawking} and proved that it is a K\"ahler Ricci-flat metric on $\C^2$ but non-flat metric. 
His example provides a nontrivial entire solution of the complex Monge-Amp\`ere equation. 
We shall construct  infinitely many solutions for Donaldson's equation \eqref{E-1}, which are nontrivial solutions in the sense that $u_{tt}$ is a constant, but $\t u, \nabla u_t$ are both not constant . However, when $n=2$, the K\"ahler metrics corresponding to these examples are  the Euclidean metric on $\C^2$.   
We shall  prove a Liouville type theorem for Donaldson's equation \eqref{E-1}, which says $u_{tt}$ has to be a constant provided some restriction on $u_{tt}$. Our proof relies on a transformation introduced by Donaldson \cite{Donaldson}. We then ask if all solutions of \eqref{E-1} satisfy that $u_{tt}$ is a constant;  this would characterize all entire solutions of \eqref{E-1}  if  confirmed.

\vspace{3mm}

{\bf Acknowledgement:} I am grateful to Prof. Xiuxiong Chen and Prof. Jingyi Chen for constant support and encouragements. I  benefit from conversations with Prof. Pengfei Guan about the complex Monge-Amp\`ere equations and I would like to thank him.   I would also like to thank the referee for valuable suggestions and comments, and for spotting a minor error in the computation of the examples in Section 2 in a previous version of the paper.

\section{Examples of entire solutions}
In this section we shall construct infinitely many nontrivial solutions of \eqref{E-1} and \eqref{E-3}.  
First we consider \eqref{E-1}.  Let $u_{tt}=2a$  for some $a>0$ and $u(0, x)=g(x), u_t(0, x)=b(x)$. Then
\begin{equation}\label{E-2-1}
u(t, x)=at^2+tb(x)+g(x).
\end{equation}
If $u$ solves \eqref{E-1}, then 
\[
2a(t\t b+\t g)-|\nabla b|^2=1.
\]
It follows that
\[
\t b=0, \t g=\frac{1}{2a}\left(1+|\nabla b|^2 \right). 
\]
So we shall construct the examples as follows. 
Let $b=b(x_1, x_2, \cdots, x_n)$ be a harmonic function in $\R^n$. Define
\[
h(x)=(1+|\nabla b|^2)/2a.
\]
Consider the following equation for $g(x)$
\begin{equation}\label{E-2-3}
\t g=h(x).
\end{equation}
Note that we can write $g=b^2(x)/4a+f$ for some function $f$ such that $\t f=1$. We can summarize our results above as follows.
\begin{theo}Let $u$ be the form of \eqref{E-2-1} such that $b$ is a harmonic function and $g$ satisfies \eqref{E-2-3}, then $u$ is an entire solution of \eqref{E-1}. Moreover, any entire solution of \eqref{E-1} with $u_{tt}=const$ has the form of \eqref{E-2-1}. 
\end{theo}

When $n=2$, the examples also provide solutions of the complex Monge-Amp\`ere equation \eqref{E-3}. Actually let $u(z, w): \C^2\rightarrow \R $ be a solution of \eqref{E-3}. If $u_{z\bar z}=a$ for some constant $a>0$, it is not hard to derive that 
\begin{equation}\label{E-2-4}
u(z, w)=az\bar z+f(z, \bar z)+zb(w, \bar w)+\bar z\bar b(w, \bar w)+g(w, \bar w)
\end{equation}
such that
\[
\frac{\p^2 f}{\p z\p \bar z}=\frac{\p^2 b}{\p w\p \bar w}=0, \frac{\p^2 g}{\p w\p\bar w}=\frac{1}{a}\left(1+\left|\frac{\p b}{\p \bar w}\right|^2\right).
\]
 However these examples are all trivial solutions of the complex Monge-Amp\`ere equation in the sense that the corresponding K\"ahler metrics are flat.

\section{A Liouville Type Theorem}
In this section we shall prove a Liouville type result for solutions of \eqref{E-1}.  We shall describe a transformation  introduced by Donaldson \cite{Donaldson}, which relates the solutions of \eqref{E-1} with harmonic functions.  Using this transformation, Theorem \ref{T-2} follows from the standard Liouville theorem for positive harmonic functions.  
\begin{theo}\label{T-2}Let $u$ be a solution of \eqref{E-1} with $u_{tt}>0$. For any $x\in \R^n$, if $u_{tt}(t, x)dt^2$ defines a complete metric on $\R\times \{x\}$, then 
$u_{tt}$ is a constant. In particular, it has the form of \eqref{E-2-1} such that $b$ is a harmonic function and $g$ satisfies \eqref{E-2-3}.
\end{theo}
\begin{proof}
For any $x$ fixed, let $z=u_t(t, x)$. Then $\Phi: (t, x)\rightarrow (z, x)$ gives a transformation since $u_{tt}>0$ and the Jacobian of $\Phi$ is always positive. In particular, when $u_{tt}(x, t)dt^2$ is a complete metric on $\R\times \{x\}$, $\Phi: (t, x)\rightarrow (z, x)$ is a diffeomorphism from $\R\times \R^n$ to $\R\times \R^n$.  For $x$ fixed, there exists a unique $t=t(z, x)$ such that $z=u_{t}(t, x)$. Define a function $\theta(z, x)=t(z, x).$ We claim that $\theta$ is a harmonic function in $\R\times \R^n$.
The identity $z=u_t(\theta, x)$ implies
\[
\frac{\p \theta}{\p x_i}u_{tt}+u_{tx_i}=0,  u_{tt}\frac{\p \theta}{\p z}=1.
\]
It then follows that
\[
\begin{split}
&u_{tt}\frac{\p^2 \theta}{\p x_i^2}+2u_{ttx_i}\frac{\p \theta}{\p x_i}+u_{ttt}\left(\frac{\p \theta}{\p x_i}\right)^2+u_{tx_ix_i}=0,\\
&u_{tt}\frac{\p^2 \theta}{\p z^2}+\frac{u_{ttt}}{u^2_{tt}}=0.
\end{split}
\]
We compute, if $u$ solves \eqref{E-1},
\[
\begin{split}
\t_{(z, x)} \theta&=\frac{\p^2 \theta}{\p z^2}+\sum_i\frac{\p^2 \theta}{\p x_i^2}\\
&=\frac{1}{u_{tt}}\left(-\frac{u_{ttt}}{u_{tt}^2}-\t u_t+2\sum_i\frac{u_{ttx_i}u_{tx_i}}{u_{tt}}-\sum_i\frac{u_{ttt}u_{tx_i}^2}{u_{tt}^2}\right)\\
&=\frac{1}{u_{tt}}\left(-\frac{u_{ttt}}{u_{tt}^2}(1+\sum u_{tx_i}^2)-\t u_t+2\sum_i\frac{u_{ttx_i}u_{tx_i}}{u_{tt}}\right)\\
&=\frac{1}{u_{tt}}\left(-\frac{u_{ttt}\t u}{u_{tt}}-\t u_t+2\sum_i\frac{u_{ttx_i}u_{tx_i}}{u_{tt}}\right)\\
&=-\frac{1}{u_{tt}^2}\p_t (\t u u_{tt}-|\nabla u_t|^2)\\
&=0.
\end{split}
\]
On the other hand, $\p \theta/\p z=1/u_{tt}>0$. Hence $\p \theta/\p z$ is a positive harmonic function on $\R\times \R^n$. It follows that
$\p \theta/\p z$ is a constant, and so $u_{tt}$ is a constant. 
\end{proof}

One could  classify all solutions of \eqref{E-1} if one could prove that $u_{tt}$ decays to zero not too fast when $|t|\rightarrow \infty$ such that $u_{tt}dt^2$ defines a complete metric on a line.  This motivates the following, 
\begin{pro}\label{C-1}Do all solutions of \eqref{E-1} with $u_{tt}>0$ satisfy $u_{tt}=const$? 
\end{pro}

\bibliographystyle{amsplain}

\end{document}